\documentclass{llncs}

\usepackage{cite}
\usepackage{graphicx}
\usepackage{amsmath}
\usepackage{amsfonts}
\usepackage{amssymb}
\usepackage{color}
\usepackage{latexsym}\normalsize\large\LARGE\LARGE\normalfont\Large\large
\usepackage{algorithm}
\usepackage[noend]{algorithmic}
\usepackage{soul}

\usepackage{ textcomp }
\usepackage{ dsfont }

\newcommand{\B}[0]{{\cal B}}
\newcommand{\C}[0]{{\cal C}}
\newcommand{\D}[0]{{\cal D}}

\newcommand{\T}[0]{{\cal T}}

\newcommand{\old}[1]{{}}

{\bfseries}{\itshape}
\newtheorem{notn}[theorem]{Notation}{\bfseries}{\itshape}
\newtheorem{observation}[theorem]{Observation}
\newtheorem{defi}[theorem]{Definition}
\newtheorem{clm}[theorem]{Claim}
\newtheorem{lem}[theorem]{Lemma}
\newtheorem{cor}[theorem]{Corollary}
\newtheorem{rem}[theorem]{Remark}

\begin{document}
\title{
				Reconstruction of the Path Graph
						}

\author{Chaya Keller
			\inst{1}
			\thanks{The work of the first author was partially supported by by Grant 635/16 from the Israel Science Foundation, by the Shulamit Aloni Post-Doctoral Fellowship of the Israeli Ministry of Science and Technology, by the Kreitman Foundation Post-Doctoral Fellowship
and by the Hoffman Leadership and Responsibility Program of the Hebrew University.}
	\and
					Yael Stein\inst{2}
					\thanks{The work of the second author was partially supported by the Lynn and William Frankel Center for Computer Science and by grant 680/11 from the Israel Science Foundation (ISF).}
}  		
  		
\authorrunning{C. Keller and Y. Stein} 
\institute{Department of Mathematics, Ben-Gurion University of the Negev
	Beer-Sheva 84105, Israel
	\and
	Department of Computer Science, Ben-Gurion University of the Negev
  Beer-Sheva 84105, Israel
	}

\index{Example, Anton}
\index{Veryknown, Sylvia}


\maketitle
\thispagestyle{plain}

\begin{abstract}
Let $P$ be a set of $n \geq 5$ points in convex position in the plane.
The path graph $G(P)$ of $P$ is an abstract graph whose vertices are non-crossing spanning paths
of $P$, such that two paths are adjacent if one can be obtained from the other
by deleting an edge and adding another edge.

We prove that the automorphism group of $G(P)$ is isomorphic to $D_{n}$, the dihedral group of
order $2n$. The heart of the proof is an algorithm that first identifies the vertices of $G(P)$ that correspond to boundary paths of $P$, where the identification is unique up to an automorphism of $K(P)$ as a geometric graph, and then identifies (uniquely) all edges of each path represented by a vertex of $G(P)$. The complexity of the algorithm is $O(N \log N)$ where $N$ is the number of vertices of $G(P)$.
\end{abstract}


\section{Introduction}

A \emph{geometric graph} is a graph whose vertices are a finite set of points in general position in the plane, and whose edges are closed segments connecting distinct points.
We consider the complete convex geometric graph $K(P)$, in which the vertex set is a convex set $P$ of $n$ points in the plane, and the edges are all segments connecting pairs of vertices. Without loss of generality we will henceforth assume that $P$ is the vertex set of a regular $n$-gon.

\begin{defi}
Let $P$ be a set of $n$ points in the plane. The \emph{path graph} $G(P)$ is defined
as follows. The vertices of $G(P)$ are the simple (i.e., non-crossing) spanning paths
of $K(P)$. Two such vertices are adjacent in $G(P)$ if they differ in exactly two edges, i.e.,
if one can be obtained from the other by deleting an edge and adding another edge.
\end{defi}

The path graph was introduced in 2001 by Rivera-Campo and Urrutia-Galicia~\cite{RU01} who showed
that when $P$ is in convex position, $G(P)$ is Hamiltonian. Following~\cite{RU01}, several works
studied $G(P)$ in the convex case. Akl et al.~\cite{AIM07} showed that $|V(G(P))|=n2^{n-3}$ and that
$\mathrm{diam}(G(P)) \leq 2n-5$. Chang and Wu~\cite{CW09} determined the diameter exactly, showing
that $\mathrm{diam}(G(P)) = 2n-5$ for $n=3,4$ and $\mathrm{diam}(G(P)) = 2n-6$ for $n \geq 5$.
Fabila-Monroy et al.~\cite{FFHHUW09} showed that the chromatic number of $G(P)$ is $n$.
Wu et al.~\cite{WCPW11} presented algorithms for generating plane spanning paths
efficiently. The general (i.e., non-convex) case is less-studied, and it is not known
even whether $G(P)$ is connected for all $P$ (see~\cite{AIM07}).

The study of $G(P)$ evolved from the study of the \emph{geometric tree graph} $\T(P)$ which has all
non-crossing spanning trees of $P$ as its vertices, and two vertices are adjacent in $G(P)$ if they
differ in exactly two edges. Defined by Avis and Fukuda~\cite{AvisFukuda} as the geometric
counterpart of the classical \emph{tree graph}~\cite{Cummins}, $\T(P)$ was studied in quite a few
works, both in the convex and in the general case (e.g.,~\cite{AAH02,AR07,GNT00,Hernando1,Hernando,KP15+}).

Some of the central results on $\T(P)$, such as Hamiltonicity and upper/lower bounds on the diameter
(see~\cite{AvisFukuda,Hernando1}) already have counterparts for $G(P)$ (proved in~\cite{RU01,AIM07,CW09}).
In this paper we establish a counterpart of another result: exact determination of the \emph{automorphism
group} in the convex case. For $\T(P)$, Hernando et al.~\cite{Hernando1} showed that $\mathrm{Aut}(\T(P))$
is $D_{n}$, the dihedral group of rotations and reflections of a regular $n$-gon. Since $\mathrm{Aut}(K(P)) \cong D_{n}$, it follows that $D_{n}$ is isomorphic to a subgroup of $\mathrm{Aut}(G(P))$.

Our main result is that there are no other automorphisms on $G(P)$.

\begin{theorem}\label{Thm:Main}
Let $P$ be a set of $n \geq 5$ points in convex position in the plane, and let $G(P)$ be its path graph.
Then $\mathrm{Aut}(G(P)) \cong D_{n}$.
\end{theorem}

The proof of Theorem~\ref{Thm:Main} relies on an algorithm that allows recovering all edges of each
path represented by a vertex of $G(P)$ (up to an automorphism of $K(P)$ as a geometric graph),
given $G(P)$ as an abstract graph. The algorithm exploits analysis of maximal cliques in $G(P)$,
following an approach pioneered by Urrutia-Galicia~\cite{Virginia-PhD}.
First, we use the structure of the max-cliques to identify an ordered subset of $n$ vertices of $G(P)$ that corresponds to the boundary paths of $P$, where the identification is fixed up to an automorphism of $K(P)$.
Then we show that once
the ordered subset is fixed, all edges of each path can be determined uniquely by examining distances between
various vertices of $G(P)$.
The running time of the algorithm is $O(N \log N)$ where $N=|V(G(P))|$, which is close to optimal, since for each of the $N$ vertices of $G(P)$ we recover the $n-1=\Theta( \log N)$ edges in the path it represents.
It should be noted that the determination of $\mathrm{Aut}(\T(P))$ in~\cite{Hernando1}
is non-constructive, and no efficient algorithm is known for full recovery of $\T(P)$.
In this sense, our result is stronger than the analogous result on $\T(P)$.
Likewise, while the technique of Urrutia-Galicia~\cite{Virginia-PhD}
was used in several previous works, this is the first time it is used for complete recovery of $G(P)$, thus solving completely
a natural \emph{graph reconstruction} problem (see, e.g.,~\cite{Bondy} for a definition and survey of
reconstruction problems).

The paper is organized as follows. Hereinafter, we present notations and a simple observation
used throughout the paper. In Section~\ref{sec:maxclique} we study the structure of maximal cliques in $G(P)$. In
Section~\ref{sec:main} we prove the main theorem. We conclude the paper with a complexity analysis, in
Section~\ref{sec:analysis}, and a few open problems.

\subsection*{Notations}
\label{sec:notations}

In this section we present notations and simple observations that will be used in the sequel.

\medskip \noindent Throughout the paper, $P$ is a set of points in convex position in the plane.
The edges of $K(P)$, the complete geometric graph on $P$, are divided into two classes: $n$ \emph{boundary
edges} of $\mathrm{Conv}(P)$ and ${{n}\choose{2}}-n$ \emph{diagonals}, i.e., edges internal to $\mathrm{Conv}(P)$.
We denote the set of boundary edges by $\B(P)$, and say that $x,y \in P$ are \emph{neighboring} if $(x,y) \in \B(P)$.
An automorphism of $K(P)$ as a geometric graph is an automorphism of $K(P)$ as an abstract graph that, in
addition, maps crossing edges into crossing edges and non-crossing edges into non-crossing edges.

As defined above, $G(P)$ denotes the (non-crossing) spanning path graph of $P$.
For $v \in V(G(P))$, $P(v)$ denotes the path represented by $v$.
For the sake of convenience, we sometimes use the term $P(v)$ also for the edge-set of the path represented by $v$.
We stress that we usually denote this edge-set by $v$; the notation $P(v)$ is used for it only in places when the meaning
is clear from the context.

The set of boundary edges of $P(v)$, that is,
$P(v) \cap \B(P)$, is denoted by $\B(v)$. The set of diagonals of $P(v)$ is denoted by $\D(v)=P(v) \setminus \B(v)$.
$P(v)$ is called a \emph{boundary path} if all its edges are boundary edges.
We denote the set of vertices of $G(P)$ that represent boundary paths by $\mathcal{B}$. Note that while $\B(v)$ denotes
the boundary edges of a specific path, $\B$ denotes a subset of the vertices of $G(P)$.

For any graph $G$, the distance between vertices $x,y,$ denoted $\mathrm{dist}(x,y)$, is the shortest length of
a path in $G$ from $x$ to $y$. The distance of a vertex from a set $\C$ of vertices is defined as $\mathrm{dist}(x,\C) =
\min_{y \in \C} \mathrm{dist}(x,y)$. The \emph{degree} of a vertex $v$ in a graph $G$ is the number of edges of $G$ that emanate from $v$, and is denoted by $\mathrm{deg}_G(v)$. A vertex is called a \emph{leaf} if its degree is $1$. An edge is called a
\emph{leaf edge} if one of its endpoints is a leaf. A vertex that is not a leaf is called an \emph{internal vertex}.

\medskip \noindent We use the following simple observation on the structure of simple spanning paths of $P$.

\begin{observation}\label{Obs:Basic}
Let $S$ be a simple spanning path of a set $P$ of points in convex position in the plane. Then:
\begin{enumerate}
\item Both leaf edges of $S$ are boundary edges.

\item If $S$ is not a boundary path, then its leaves cannot be neighboring vertices of the boundary.
\end{enumerate}
\end{observation}

\noindent The easy proof of the observation is omitted.

\section{Maximal Cliques in $G(P)$}
\label{sec:maxclique}

The reconstruction of the paths represented by vertices of $G(P)$ requires a fulcrum to start with.
Our fulcrum is understanding of the maximal cliques in $G(P)$. We note that the approach of exploiting
maximal cliques for this purpose was pioneered by Urrutia-Galicia~\cite{Virginia-PhD} in the context
of geometric tree graphs, and used recently in~\cite{KP15+}.

\begin{defi}
A \emph{max-clique} in a graph $G$ is a maximal (with respect to inclusion) clique
included in $G$. Since a max-clique is a complete graph on its vertex set, we shall
identify a max-clique with its set of vertices.
\end{defi}

We start our discussion of max-cliques with purely combinatorial considerations that do not exploit
the geometric nature of the problem. Let $u,v \in V(G(P))$ be neighbors.
We denote by $\bar{u}$ and $\bar{v}$ the sets of edges of $P(u)$ and $P(v)$, respectively.
Clearly, $|\bar{u} \cup \bar{v}|=n$, $|\bar{u} \cap \bar{v}| = n-2$, and $|\bar{u} \triangle \bar{v}|=2$.
Let $w$ be a common neighbor of $u$ and $v$ in $G(P)$ (if it exists).
Since $|P(w)|=n-1$ and $P(w)$ differs from each of
$P(v),P(u)$ in exactly two edges, there are exactly two possibilities for $\bar{w}$:

\begin{enumerate}
\item $\bar{w} \cap (\bar{u} \triangle \bar{v}) = \emptyset$, and then $(\bar{u} \cap \bar{v}) \subset \bar{w}$, i.e., $\bar{w}$ consists of $\bar{u} \cap \bar{v}$ plus an
additional edge,

\item $(\bar{u} \triangle \bar{v}) \subset \bar{w}$, and then $\bar{w} \subset (\bar{u} \cup \bar{v})$, i.e., $\bar{w}$ consists of all edges of $\bar{u} \cup \bar{v}$ except
for one edge of $\bar{u} \cap \bar{v}$.
\end{enumerate}

Note that if $w$ satisfies (1), then each other common neighbor of $u,v,w$ (i.e., each other element of the max-clique
that contains $u,v,w$) also satisfies (1). Conversely, each $w,w'$ that both satisfy (1) are neighbors. The same holds
with (1) replaced by (2). Hence, we obtain:

\begin{cor}
\label{cor:max}
Each edge $(u,v) \in E(G(P))$ is contained in at most two max-cliques:

\begin{itemize}

\item An \emph{intersection max-clique}
$$I(u,v) = \{w \in V(G(P)): \bar{w}=(\bar{u} \cap \bar{v}) \cup \{e\}, \mbox{for some } e \not \in \bar{u} \cap \bar{v}\},$$

\item A \emph{union max-clique}
$$U(u,v)=\{w \in V(G(P)): \bar{w}=(\bar{u} \cup \bar{v}) \setminus \{e\}, \mbox{for some } e \in \bar{u} \cup \bar{v}.$$

\end{itemize}
In addition, given three vertices in a max-clique in $G(P)$, they uniquely determine its type.

Note that by this definition, $u,v\in I(u,v)$ and $u,v\in U(u,v)$.
\end{cor}

\begin{rem}
For ease of notation, we call $I(u,v)$ an ``intersection max-clique'' even if $I(u,v)=\{u,v\}$, i.e., it contains only
two vertices. This is a slight abuse of notation, since in such a case, $I(u,v)$ may be properly contained in $U(u,v)$,
and thus, not be a max-clique by the definition above. Similarly, we call $U(u,v)$ a ``union max-clique'' even if
$|U(u,v)|=2$.
\end{rem}

\noindent Now we present a geometric characterization of the two types of max-cliques.

\medskip \noindent \textbf{Intersection max-clique.} Given two neighbors $u,v \in V(G(P))$, the intersection $\bar{u} \cap \bar{v}$
can be viewed as a disjoint union of two simple paths $(x_1,x_2,\ldots,x_k)$, $(y_1,y_2,\ldots,y_{\ell})$, where $\{x_i\},\{y_j\} \subset P$,
$1 \leq k,\ell \leq n-1$, and $k+\ell=n$. Each element in $I(u,v)$, including $u$ and $v$, is obtained from $\bar{u} \cap \bar{v}$ by
adding one of the four edges $(x_k,y_\ell),(x_k,y_1),(x_1,y_{\ell}),(x_1,y_1)$, such that the resulting path is non-crossing.
If none of these four edges crosses edges of $\bar{u} \cap \bar{v}$, we get $|I(u,v)|=4$ (see Figure~\ref{fig:inclique-4}). If (w.l.o.g.) $(x_1,y_1)$
crosses $e \in \bar{u} \cap \bar{v}$, w.l.o.g. $e=(x_j,x_{j+1})$, then $(x_1,y_{\ell})$ also crosses $e$ (since all the path
$(y_1,\ldots,y_{\ell})$ lies on the same side of $e$) and then $|I(u,v)|=2$ (see Figure~\ref{fig:Uclique-2}).

\begin{figure}[tb]
 \centering
    \centering
  			\includegraphics[width=.4\columnwidth, page=3]{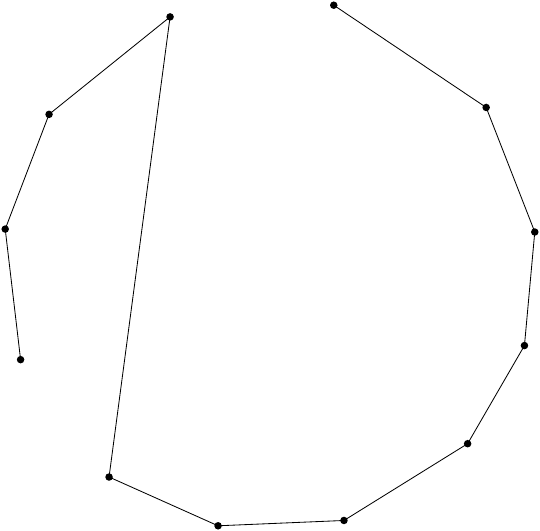}
 \hspace{0.3cm}
    \centering
 				\includegraphics[width=.4\columnwidth, page=4]{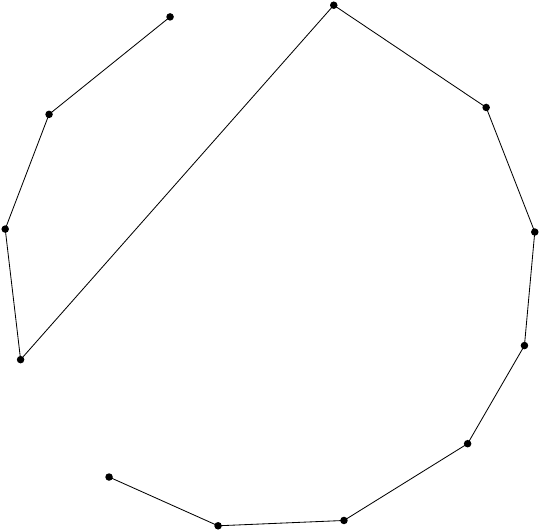}
 \caption{These two Hamiltonian paths are neighbors in $G(P)$, and are included in an \textit{intersection-clique} of size $4$, and in a \textit{union-clique} of size~2 (only these two paths).}
 \label{fig:inclique-4}
\end{figure}

\begin{figure}
 \centering
    \centering
  			\includegraphics[width=.4\columnwidth, page=5]{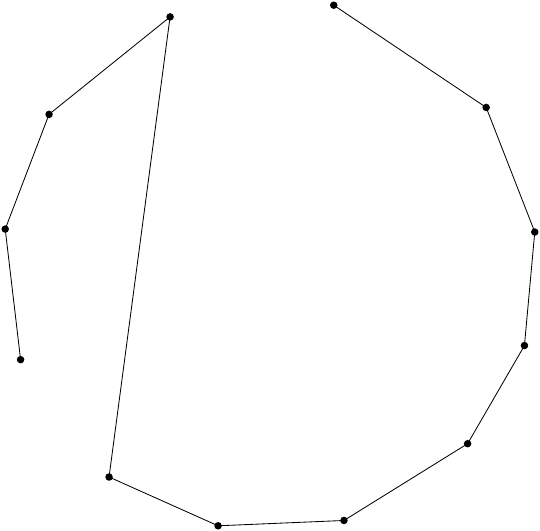}
 \hspace{0.3cm}
    \centering
 				\includegraphics[width=.4\columnwidth, page=6]{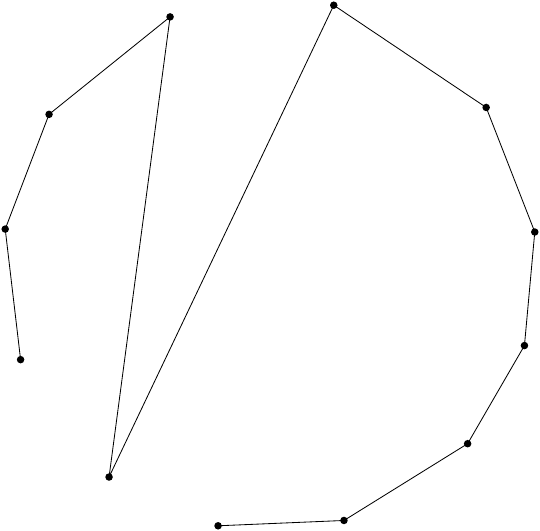}
 \caption{These two Hamiltonian paths are neighbors in $G(P)$, and generate a maximal clique of size $2$. I.e., the \textit{intersection-clique} is identical to the \textit{union-clique}, both of size $2$.}
 \label{fig:Uclique-2}
\end{figure}

\medskip \noindent \textbf{Union max-clique.} Given two neighbors $u,v \in V(G(P))$, where $P(u)=(x_1,x_2,\ldots,x_n)$, it is easy
to see that $\bar{u} \cup \bar{v}$ is either of the form $\bar{u} \cup \{(x_1,x_j)\}$ where $2<j \leq n$, or of the form
 $\bar{u} \cup \{(x_j,x_n)\}$ where $1 \leq j<n-1$. Assume w.l.o.g. the former holds.

Each element in $U(u,v)$, including $u$ and $v$,
is obtained from $\bar{u} \cup \bar{v}$ by removing an edge, such that the resulting graph is a non-crossing spanning path.
We distinguish between two cases:

\begin{itemize}
\item If $\bar{u} \cup \bar{v} = \bar{u} \cup \{(x_1,x_n)\}$ then the edge $(x_1,x_n)$ crosses at most one edge of $\bar{u}$ (as
otherwise, $P(v)$ cannot be non-crossing). If $(x_1,x_n)$ crosses $e \in \bar{u}$, then we must have $\bar{v}=\bar{u} \cup \{(x_1,x_n)\}
\setminus \{e\}$ and $|U(u,v)|=2$. If $(x_1,x_n)$ does not cross any edge of $\bar{u}$ then $\bar{u} \cup \bar{v}$ is the boundary
of $\mathrm{Conv}(P)$, and thus $|U(u,v)|=n$ (see Figure~\ref{fig:Uclique-n}).

\item If $\bar{u} \cup \bar{v} = \bar{u} \cup \{(x_1,x_j)\}$ for $j<n$ then we must have $\bar{v}=\bar{u} \cup \{(x_1,x_j)\}
\setminus \{(x_{j-1},x_j)\}$ and then $|U(u,v)|=2$.
\end{itemize}

\begin{figure}[bt]
 \centering
    \centering
  			\includegraphics[width=.4\columnwidth, page=1]{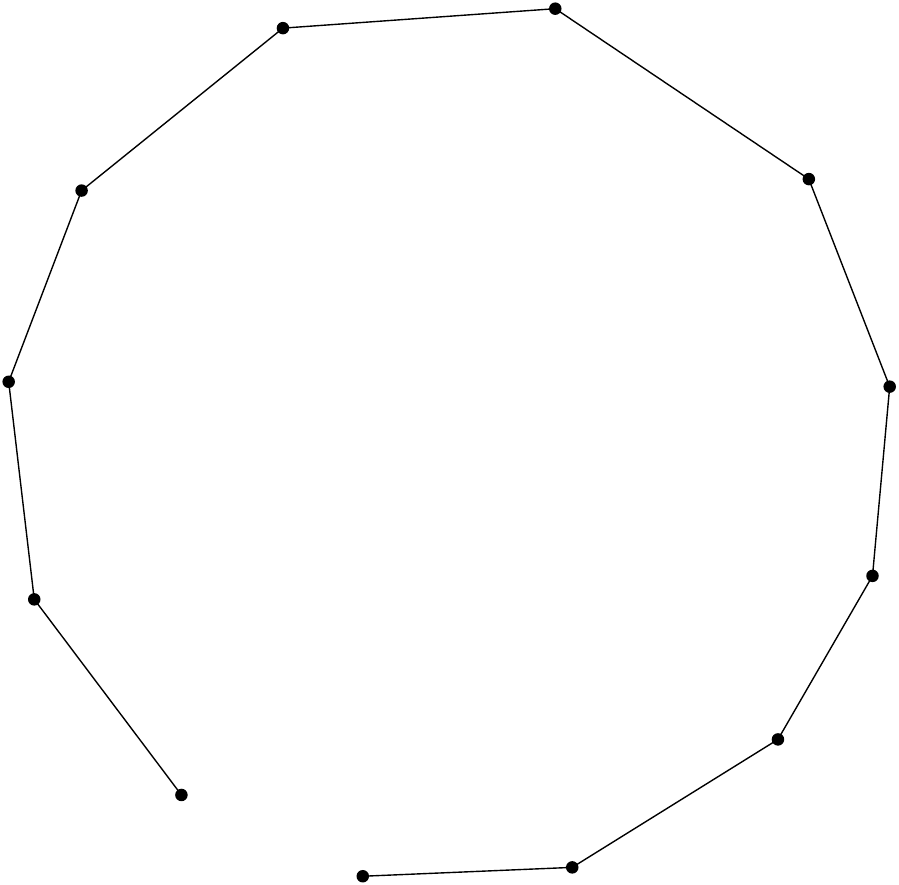}
 \hspace{0.3cm}
    \centering
 				\includegraphics[width=.4\columnwidth, page=2]{Uclique_n.pdf}
 \caption{These two Hamiltonian paths are neighbors in $G(P)$, and are included in an \textit{intersection-clique} of size $2$  and in a \textit{union-clique} of size $n$.}
 \label{fig:Uclique-n}
\end{figure}

\medskip \noindent Summarizing the above, we have the following.

\begin{cor} \label{cor:clique size}
Let $P$ be a set of $n \geq 5$ points in convex position in the plane, and let $G(P)$ be the path graph of $P$.
Then:
\begin{itemize}
\item Each intersection max-clique of $G(P)$ is either of size 2 or 4.

\item Among the union max-cliques, all are of size 2 except for a single max-clique of size $n$, in which each vertex represents
a boundary path that contains all edges of $\B(P)$ except for one.
\end{itemize}
\end{cor}

\section{The Automorphism Group of $G(P)$}
\label{sec:main}

In this section we show that given $G(P)$ as an abstract graph, we can recover all edges of
each path represented by a vertex of $G(P)$, up to an automorphism of $K(P)$ as a geometric
graph. This clearly implies that $\mathrm{Aut}(G(P))$ is the dihedral group of order $2n$.
(For sake of completeness, we prove this easy implication at the end of the section.)

\medskip \noindent The proof proceeds in three steps:
\begin{enumerate}
\item We detect all vertices of $G(P)$ that represent boundary paths. Namely, we find an ordered subset of $n$
vertices of $G(P)$ with a bijection between them and the boundary edges of $K(P)$, fixed up to an automorphism of $K(P)$ as a geometric graph.

\item We divide all vertices of $G(P)$ into levels according to their distance from the family of
boundary paths, and use the identification of boundary paths to recover uniquely all boundary
edges of each path represented by a vertex of $G(P)$.

\item We use the relation between vertices at adjacent levels to recover uniquely all
diagonals of each path represented by a vertex of $G(P)$.
\end{enumerate}

\subsection{Identification of a ``copy'' of the boundary of $\mathrm{Conv}(P)$ inside $G(P)$}
\label{sec:sub:boundary}

As shown in
Section~\ref{sec:maxclique},
$\mathcal{B}$, the set of vertices of $G(P)$ that represent boundary paths,
is a unique max-clique of size $n$ in $G(P)$. This is already a sufficient identification of $\B$ as a set, (i.e., without order),  but for sake of obtaining an efficient algorithm for the reconstruction problem, we suggest here an alternative identification of $\B$ as a set,
based on the fact that $\B$ is exactly the set of vertices of maximum degree in $G(P)$:

\begin{clm}\label{claim:deg}
For any $v \in \B$, $\mathrm{deg}_{G(P)}(v)=3n-7$, and for any $u \in V(G(P)) \setminus \B$, $\mathrm{deg}_{G(P)}(u)<3n-7$.
\end{clm}

\begin{proof}
Let $v \in \B$. Any neighbor of $v$ in $G(P)$ represents a simple Hamiltonian path, obtained from $P(v)$ by deleting an edge and replacing it with another edge. If the deleted edge is a leaf edge of $P(v)$, only one neighbor of $v$ is obtained, and if the deleted edge is an internal edge of $P(v)$, then three neighbors of $v$ are obtained. Indeed, note that deletion of an internal edge transforms $P(v)$ into two boundary paths of total length $n-2$. There are four options to add an edge that will connect these paths into a single Hamiltonian path. Since $P(v)$ is a boundary path, all of them constitute simple paths. Exactly one of them is the original path $P(v)$, and so, deletion of any internal edge contributes 3 neighbors of $v$.
Hence, $$\mathrm{deg}_{G(P)}(v)=3(n-3)+2=3n-7.$$

On the other hand, let $u \in V(G(P)) \setminus \B$. By the definition of $\B$, $P(u)$ contains a diagonal $e$, and the two endpoints of $P(u)$ are located on different sides of $e$. As above, any neighbor of $u$ in $G(P)$ represents a simple Hamiltonian path, obtained from $P(u)$ by deleting an edge and replacing it with another edge. If the deleted edge is a leaf edge of $P(u)$, then after the deletion we are left with
a boundary path of length $n-2$ and an isolated vertex. The new edge replacing the removal boundary edge has to connect the isolated vertex to one of the leaves of the boundary path. However, for one of the two leaves, this edge crosses $e$ and so cannot be added. For the other leaf, we return to the original path $P(u)$.
Hence, $u$ has no neighbor in $G(P)$ that is obtained by deleting a leaf edge of $P(u)$. Furthermore, by deleting an internal edge of $P(u)$, at most three neighbors of $u$ can be obtained, as above, and thus $\mathrm{deg}_{G(P)}(v)<3n-7$.
\qed
\end{proof}

Now, after identifying $\B$ as a subset of $V(G(P))$, note that each $v \in \mathcal{B}$ can be represented
by the unique boundary edge of $P$ that is not contained in $P(v)$, which we denote by $e_v$.
In order to determine (to the extent possible) what is the boundary edge $e_v$ that corresponds to $v$,
and thus to identify a copy of the set of boundary edges of $K(P)$ in $G(P)$, we use the following observation.

\begin{observation}\label{Obs:Sharing a vertex}
Let $u,v \in \B$. The edges $e_u$ and $e_v$ share a vertex if and only if $(u,v)$ is not contained in a maximal clique of size $4$ in $G(P)$.
\end{observation}

\begin{proof}
If $e_u \cap e_v = \{x\}$ then $P(u) \cap P(v)$ is a two-component forest in which one component is a
boundary path $S$ of length $n-2$ and the other component is $\{x\}$. In such a case, each element of $I(u,v)$
is obtained by adding to $P(u) \cap P(v)$ an edge that connects $x$ to an endpoint of $S$. Hence, the
only elements of $I(u,v)$ are $u$ and $v$.
On the other hand, from Corollary~\ref{cor:clique size}, $|U(u,v)|\neq 4$, and therefore $(u,v)$ is not contained in any maximal clique of size $4$.

If $e_u \cap e_v = \emptyset$, then $P(u) \cap P(v)$ is a two-component forest in which the components are
boundary paths $S,S'$ of length $\geq 1$, i.e., contain at least two vertices of $P$.
In such a case, there are four different edges connecting an endpoint
of $S$ to an endpoint of $S'$, and hence, $|I(u,v)|=4$.
\qed
\end{proof}

Observation~\ref{Obs:Sharing a vertex} allows identifying a ``copy'' of the boundary of $\mathrm{Conv}(P)$ in $G(P)$, as follows. Define a graph
whose vertex set is $\B$, such that $v,w \in \B$ are connected by an edge if and only if $e_v,e_w$ share a single
vertex. Clearly, the resulting graph is a cycle of length $n$. Identify this cycle with the boundary of $\mathrm{Conv}(P)$,
in such a way that each boundary edge $e$ corresponds
to some $v \in \B$, and each $x \in P$ corresponds to a pair $\{v,w\}$ such that $e_v \cap e_w = \{x\}$. Note that
the identification is fixed only up to an automorphism of $K(P)$ as a geometric graph. However, this is clearly
best possible, since any automorphism of $K(P)$ induces an automorphism of $G(P)$.

\subsection{Recovery of the boundary edges of each path}
\label{sec:sub:levels}

We divide the vertices of $G(P)$ into {\it levels} according to the number of diagonals they contain.
\begin{notn}
For $v \in V(G(P))$, the \emph{level} of $v$ is $\ell(v)=|\D(v)|$.
\end{notn}

The following observation shows that the levels of the vertices can be recovered from $G(P)$.
This observation was made in Lemma~3.2 of~\cite{AIM07} in order to show that the diameter of $G(P)$ is
at most $2n-5$. For sake of completeness, we also give a simple proof here.

\begin{observation}\label{Obs:Level-recovery}
For each $v \in V(G(P))$, we have $\ell(v)=\mathrm{dist}(v,\B)$.
\end{observation}

\begin{proof}
It is clear from the definition of $\B$ that $\ell(v)=0$ if and only if $v \in \B$, and that for any $v \in V(G(P))$
we have $\mathrm{dist}(v,\B) \geq \ell(v)$. The inequality $\mathrm{dist}(v,\B) \leq \ell(v)$ will follow by
induction once we show that each $v \in V(G(P))$ with $\ell(v)>0$ has a neighbor $u \in V(G(P))$ with $\ell(u)=\ell(v)-1$.

Consider a leaf $x$ of $P(v)$. Clearly, exactly one of the boundary edges of $K(P)$ that emanate from $x$ is
included in $P(v)$. Denote by $(x,y)$ the boundary edge that is \emph{not} included in $P(v)$. Since $\ell(v)>0$,
$y$ cannot be a leaf of $P(v)$ (see Observation~\ref{Obs:Basic}). Thus, $y$ is adjacent in $P(v)$ to $w,z$. Without loss
of generality, the points $x,w$ lie on different sides of the edge $(y,z)$ as depicted in Figure~\ref{fig:A}.
(Otherwise, $x,z$ must lie on different sides of $(y,w)$.) In such a case, $u$, defined by $P(u)=P(v) \cup \{(x,y)\}
\setminus \{(y,z)\}$, is a neighbor of $v$ in $G(P)$ that satisfies $\ell(u)=\ell(v)-1$.
\begin{figure}
 \centering
	\includegraphics[width=.4\columnwidth]{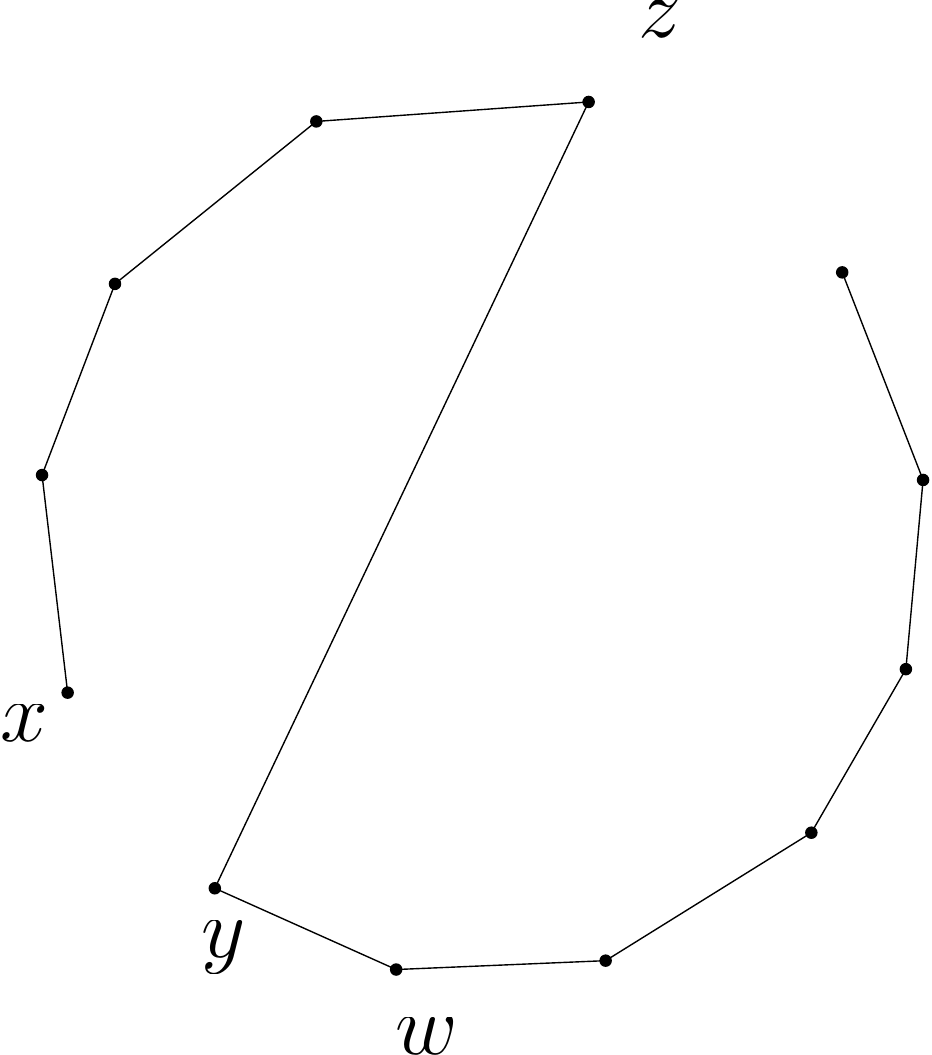}
 \caption{An illustration for the proof of Observation~\ref{Obs:Level-recovery}.}
 \label{fig:A}
\end{figure}
\qed
\end{proof}

For each $v \in V(G(P))$ of level $d$, there are exactly $d+1$ boundary edges that are not contained in $P(v)$.
The following observation shows that these edges can be recovered by observing the elements of $\B$ whose distance
from $v$ is exactly $d$. This observation follows from Lemma~5 of~\cite{CW09}. For sake of completeness, we give
its simple proof here.

\begin{observation}\label{Obs:Boundary-recovery}
Let $v \in V(G(P))$ with $\ell(v)=d$. Let $$\B(P) \setminus \B(v) = \{e_1,e_2,\ldots,e_{d+1}\}.$$ The set
$\{w \in \B: \mathrm{dist}(w,v)=d\}$ has exactly $d+1$ elements, which are the vertices of $\B$ that correspond
to the edges $e_1,e_2,\ldots,e_{d+1}$.
\end{observation}

\begin{proof}
It is clear that if $w \in \B$ and $\mathrm{dist}(w,v)=d$, then the only boundary edge not contained in $P(w)$ must be
one of $e_1,e_2,\ldots,e_{d+1}$. On the other hand, let $w \in \B$ be such that $e_i \not \in P(w)$. We claim that there
exists a path of length $d$ in $G(P)$ from $v$ to $w$. By the proof of Observation~\ref{Obs:Level-recovery},
from each $v' \in V(G(P))$ with $\ell(v')>0$ we can move to a neighbor of lower level by choosing a leaf $x$, adding
a boundary edge that emanates from it, and removing another edge. Since each such $v'$ has two leaves that are not
neighboring on $\B(P)$ (see Observation~\ref{Obs:Basic}), at each step there are two possible boundary edges that can
be added. Hence, we can construct a path in which $e_i$ is not added at any step, and thus, is missing also in the
path whose level is $0$. That final path must be $P(w)$.
\qed
\end{proof}

Since the set $\{w \in \B: \mathrm{dist}(w,v)=d\}$ can be detected in $G(P)$, Observation~\ref{Obs:Boundary-recovery}
implies that we can recover $\B(v)$ for all $v \in V(G(P))$.

\subsection{Recovery of the diagonals of each path}
\label{sec:sub:diagonals}

Our next goal is the full recovery of $P(v)$ for any path $v \in V(G(P))$, i.e., determination whether
$(x,y) \in P(v)$ or not for each $(x,y) \in E(K(P))$. We use the following observation.

\begin{observation}\label{Obs:Completion by diagonals}
Let $P_1,P_2,\ldots,P_k$ be disjoint boundary paths, possibly
including degenerate (i.e., single-vertex) paths, that cover - in the aforementioned order - all the vertices of $P$.
There are at most $k$ possible ways to extend
$P_1 \cup P_2 \cup \ldots \cup P_k$ into a simple spanning path $P(v)$ such that $\B(v) = P_1 \cup P_2
\cup \ldots \cup P_k$ by adding $k-1$ diagonals.
\end{observation}

\begin{proof}
It is easy to see that a degenerate $P_i$ cannot be an endpoint of a path $P(v)$ such that
$\B(v) = P_1 \cup P_2 \cup \ldots \cup P_k$, and that choosing an endpoint of one of the $P_i$'s to be
an endpoint of the path $P(v)$ determines $P(v)$ uniquely (i.e., leaves a single possibility to add
the $k-1$ diagonals), see Figure~\ref{fig:B}. As there are at most $2k$ such endpoints and each path has
two endpoints, at most $k$ different paths can be constructed.
\begin{figure}
\centering
	    \centering
	  			\includegraphics[width=.4\columnwidth, page=1]{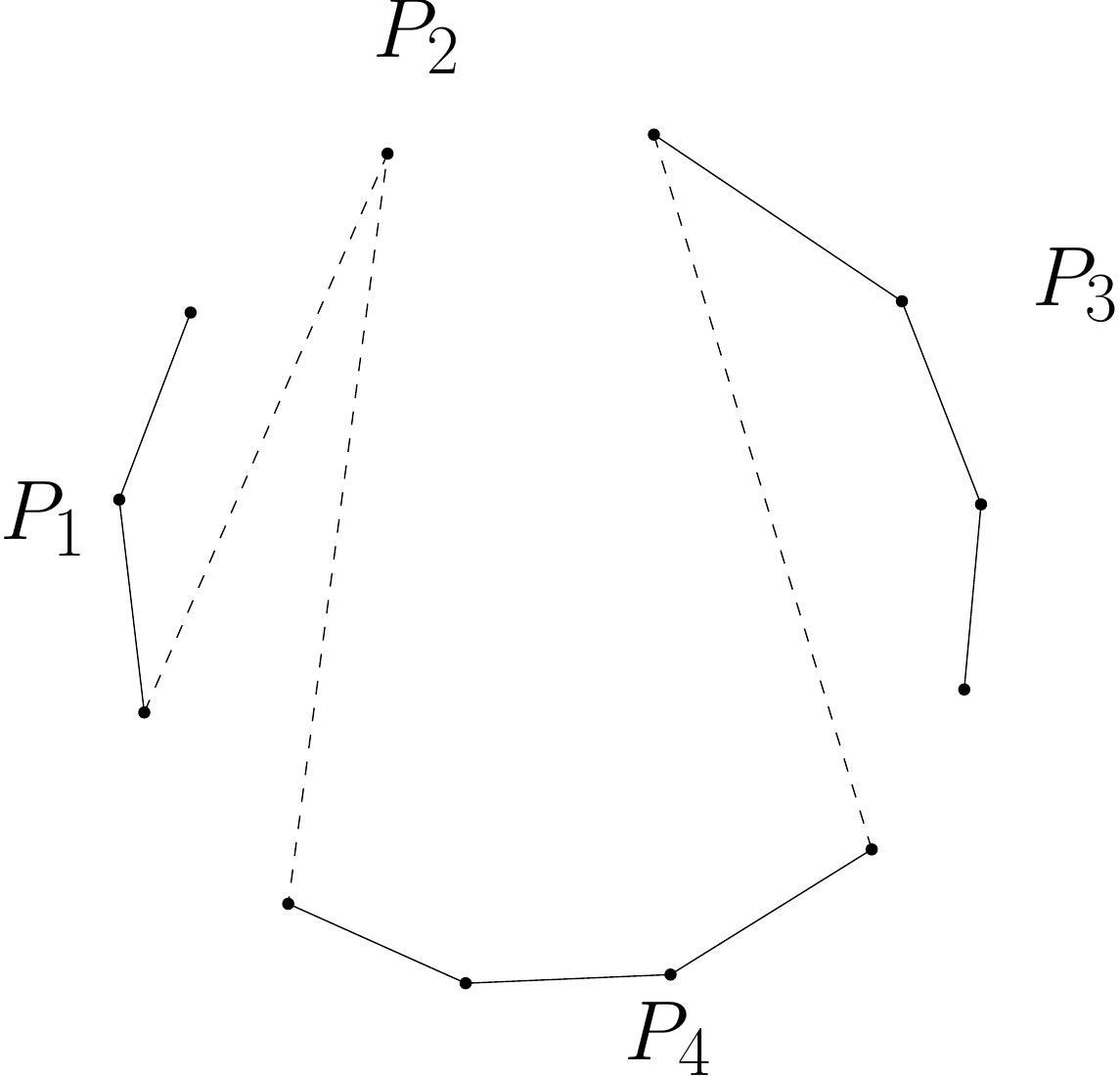}
	 \hspace{0.3cm}
	    \centering
	 				\includegraphics[width=.4\columnwidth, page=2]{chaya2.pdf}
	
 \caption{An illustration for Observation~\ref{Obs:Completion by diagonals}, where $k=4$.
 							There are two possibilities to add $3$ diagonals here (dashed).}
 \label{fig:B}
\end{figure}
\qed
\end{proof}

The determination of the diagonals is performed by induction on $\ell(v)$. The case $\ell(v)=0$ is already done,
since the boundary edges were recovered in Section~\ref{sec:sub:levels}. As the case $\ell(v)=1$ is somewhat
different from the other cases, we present it separately.

Let $v \in V(G(P))$ satisfy $\ell(v)=1$. In such a case, $\B(v)$ consists of two paths $P_1,P_2$.
Clearly, neither of them is degenerate, and at least one of them -- w.l.o.g., $P_1$ -- contains at least two
edges since $n\geq 5$ (see Figure~\ref{fig:C})\footnote{One can check easily that if $n=4$, then our main theorem does not hold, because of the symmetry between pairs of paths in level $1$.}. Denote the endpoints of $P_1$ by $a,c$ and the vertex of $P_1$ adjacent to $a$
by $b$. Furthermore, denote the endpoint of $P_2$ adjacent to $c$ by $d$, and the other endpoint of $P_2$ by $y$.

\begin{figure}[bt]
 \centering
	    \centering
	  			\includegraphics[width=.4\columnwidth, page=1]{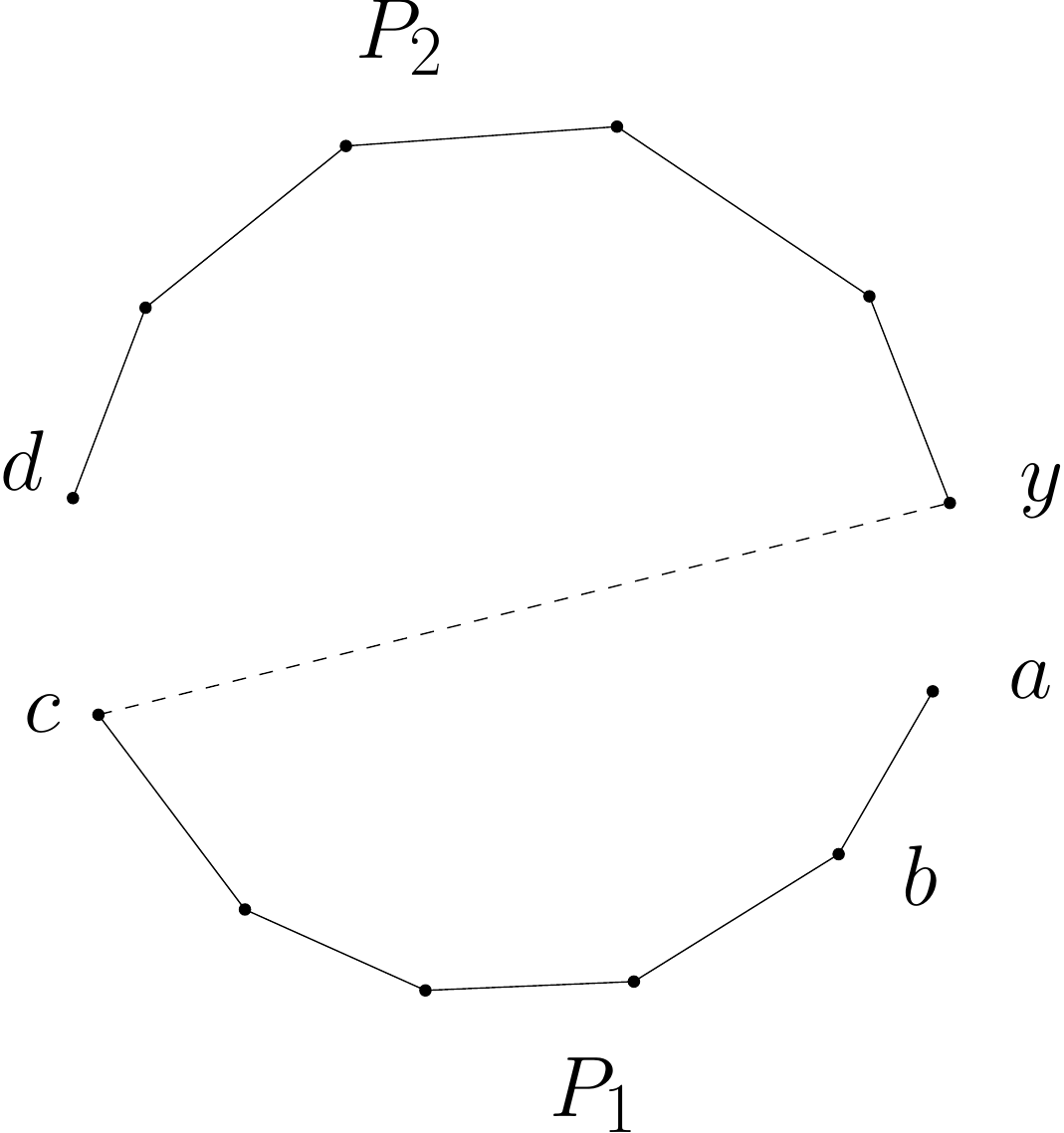}
	 \hspace{0.3cm}
	    \centering
	 				\includegraphics[width=.4\columnwidth, page=2]{chaya3.pdf}
	
 \caption{An illustration for the two cases of Observation~\ref{Obs:Level1}.}
 \label{fig:C}
\end{figure}

So far, we can recover $P_1$ and $P_2$. After they are recovered, by Observation~\ref{Obs:Completion by diagonals},
there are only two possibilities for $P(v)$: adding either $(c,y)$ or $(a,d)$. The following observation allows
distinguishing between these two cases.

\begin{observation}\label{Obs:Level1}
With the above notations, if $P(v)=P_1 \cup P_2 \cup \{(a,d)\}$ then there exists a neighbor $v'$ of $v$ such that
$\ell(v')=2$ and $(a,b) \not \in P(v')$. If $P(v)=P_1 \cup P_2 \cup \{(c,y)\}$ then there is no such neighbor.
\end{observation}

\begin{proof}
If $P(v)=P_1 \cup P_2 \cup \{(a,d)\}$, then $v'$ defined by $P(v')=P_1 \cup P_2 \cup \{(a,d)\} \cup \{(a,c)\} \setminus \{(a,b)\}$
is the desired neighbor. If $P(v)=P_1 \cup P_2 \cup \{(c,y)\}$ and a neighbor $v'$ is constructed by removing the edge $(a,b)$ (see Figure~\ref{fig:C}),the added edge must be $(a,d)$ (it must emanate from $a$ as otherwise $a$ is isolated, and the other endpoint must be $d$ as all other
vertices are already of degree $2$), and this is impossible since $(a,d)$ crosses $(c,y)$.
\qed
\end{proof}

Observation~\ref{Obs:Level1} implies that if $\ell(v)=1$ then all edges of $P(v)$ can be recovered. Assume now that
$\ell(v) := k-1 \geq 2$ and that for any $v$ with $\ell(v) \leq k-2$ we can recover all edges of $P(v)$. We show that
all edges of $P(v)$ can be recovered.

The boundary edges of $P(v)$ can be divided into $k$ (possibly degenerate) paths $P_1, P_2, \ldots, P_k$ that can be
recovered by the technique of Section~\ref{sec:sub:levels}. Once they are recovered, by Observation~\ref{Obs:Completion by diagonals},
in order to fully recover $P(v)$, it is sufficient to determine which of the endpoints of the $P_i$'s is a leaf
of $P(v)$. Note that as mentioned above, a degenerate $P_i$ cannot be a leaf of $P(v)$, and that there are at least
two non-degenerate $P_i$'s, as any spanning path has at least two boundary edges, and they lie in different $P_i$'s
unless the path is a boundary path. The leaves of $P(v)$ can be determined using the
following observation.

\begin{observation}\label{Obs:LevelBeyond1}
Let $P_i$ be non-degenerate. Denote the endpoints of $P_i$ by $a,c$, denote the endpoint of
$P_{i+1}$ adjacent to $a$ by $b$ 
(see Figure~\ref{fig:D}(a)). Then $a$ is
a leaf of $P(v)$ if and only if there exists a neighbor $v'$ of $v$ such that $\ell(v')=k-2$, $(a,b) \in P(v')$, and
$c$ is a leaf of $P(v')$.
\end{observation}

\begin{figure}  [tb]
 \centering
	    \centering
	  			\includegraphics[width=.4\columnwidth]{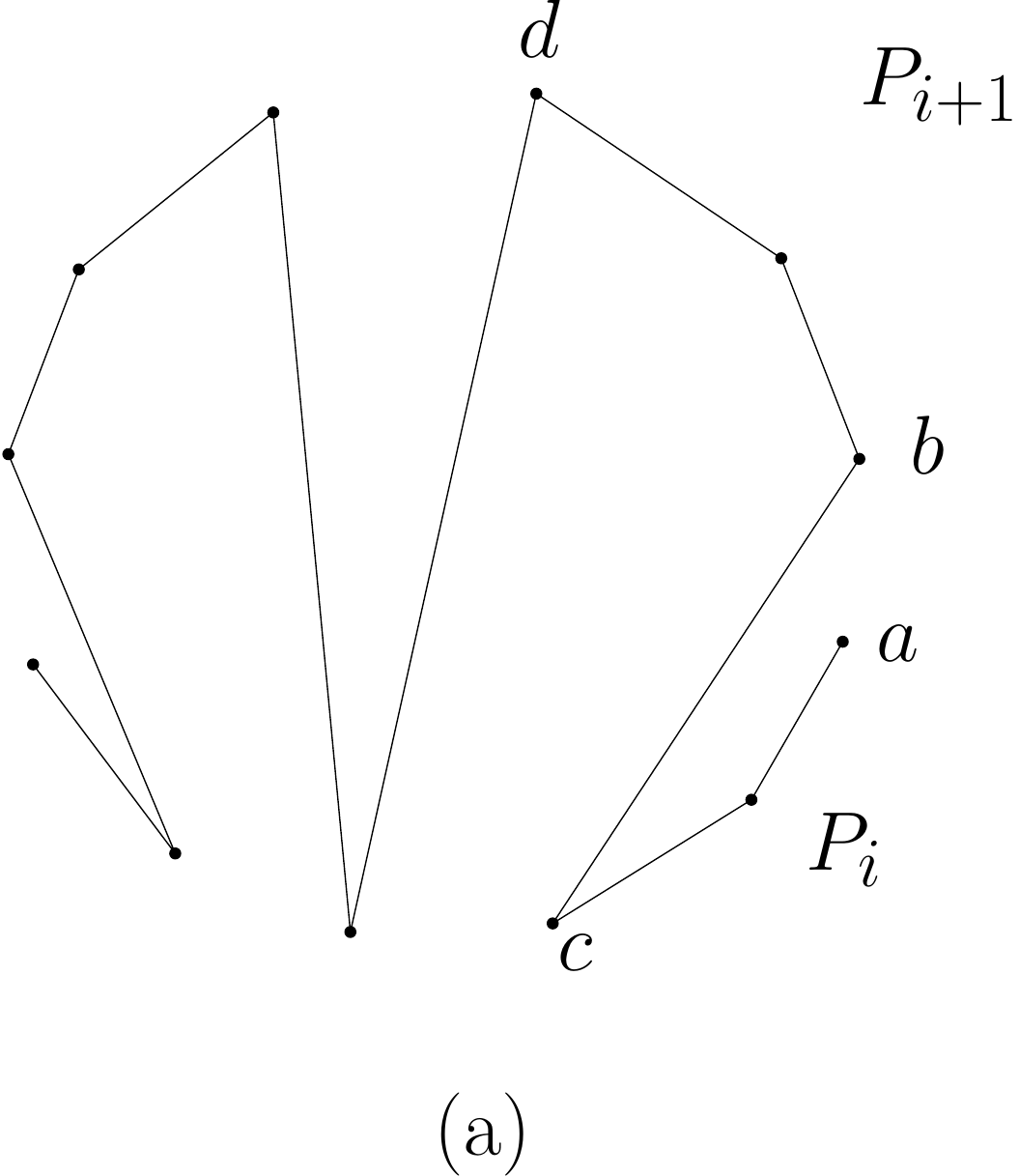}
	 \hspace{0.5cm}
	    \centering
	 				\includegraphics[width=.4\columnwidth]{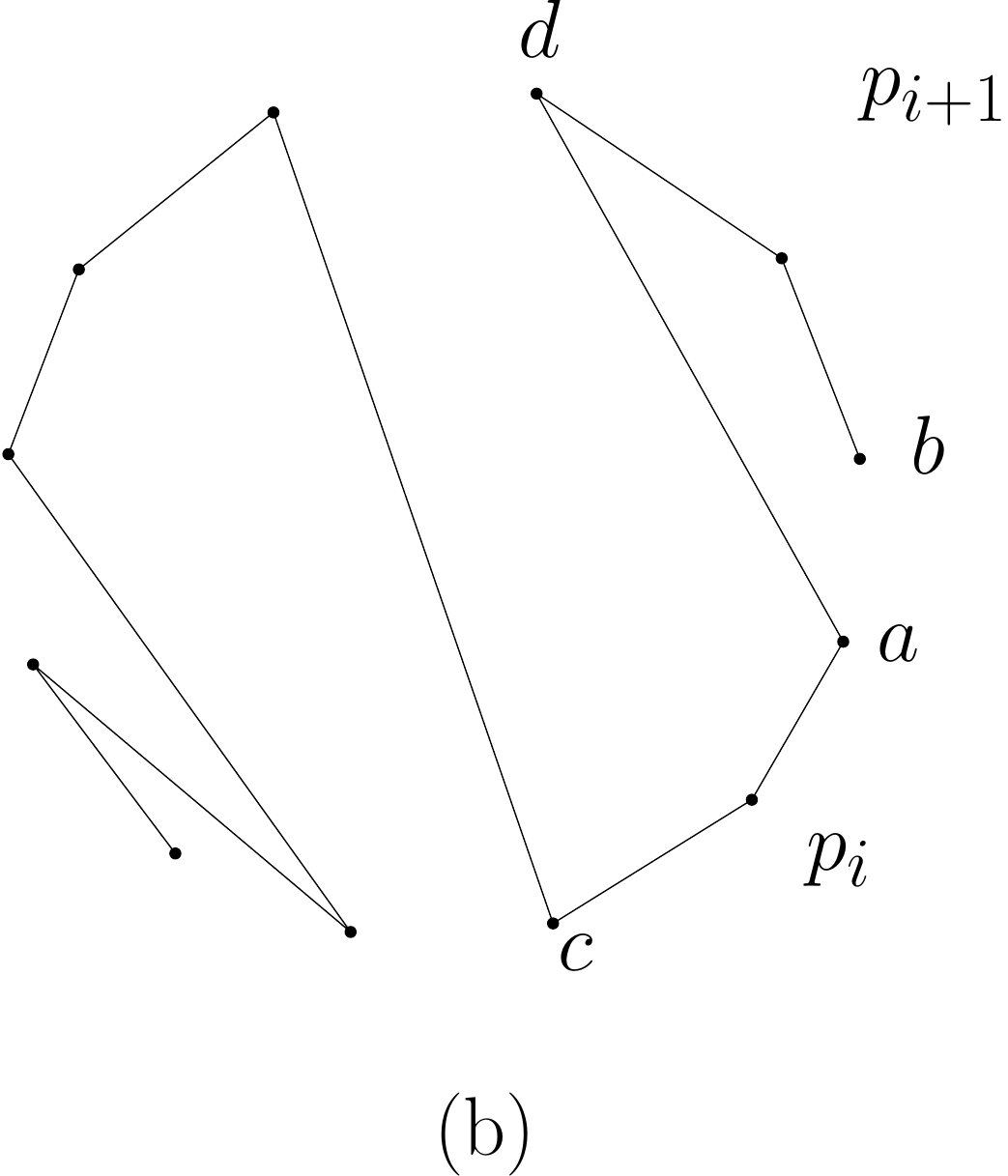}
	\caption{Illustrations for Observation~\ref{Obs:LevelBeyond1}.
							In (a), $a$ is a leaf of $P(v)$,
								while in (b) $a$ is not a leaf, but $b$ is a leaf of $P(v)$.}
	\label{fig:D}
	
\end{figure}

\begin{proof}
If $a$ is a leaf of $P(v)$, as depicted in Figure~\ref{fig:D}(a), then $(b,c) \in P(v)$. Hence, $v'$ defined by $P(v')=(P(v) \setminus \{(b,c)\}) \cup
\{(a,b)\}$ is the desired neighbor.

If both $a$ and $b$ are internal vertices of $P(v)$ then there does not exist a neighbor $v'$ with $(a,b) \in P(v')$,
since $P(v) \cup (a,b)$ contains two vertices of degree~3.

Finally, if $a$ is an internal vertex of $P(v)$ and $b$ is a leaf of $P(v)$, as depicted in Figure~\ref{fig:D}(b),
then $P_{i+1}$ is not degenerate. Denote its other endpoint by $d$.
Then the only neighbor $v'$ of $v$ such that $\ell(v')=k-2$ and $(a,b) \in P(v')$ satisfies $P(v')=(P(v) \setminus \{(a,d)\}) \cup
\{(a,b)\}$. In $P(v')$, $d$ is a leaf, and hence, $c$ cannot be a leaf of $P(v')$ as by Observation~\ref{Obs:Basic} this would
imply that $v'$ is a boundary path, contrary to the assumption $k \geq 3$. This completes the proof.
\qed
\end{proof}

Combining observations~\ref{Obs:Level1} and~\ref{Obs:LevelBeyond1}, we can recover all edges of any $v \in V(G(P))$,
by induction on $\ell(v)$.

\subsection{The automorphism group of $G(P)$ is $D_{n}$}
\label{sec:sub:automorphism}

As mentioned above, it is clear that any automorphism of $K(P)$ as a geometric graph induces an automorphism of $G(P)$.
It is well-known that $\mathrm{Aut}(K(P))=D_{n}$,
and thus, $D_{n} \hookrightarrow \mathrm{Aut}(G(P))$ (i.e., $D_{n}$ is isomorphic to a subgroup of $\mathrm{Aut}(G(P))$).

On the other hand, any automorphism of $G(P)$ must preserve the sizes of the max-cliques, and in particular, preserve
the set $\B$. Moreover, it must preserve the information whether for $v,w \in \B$, the edges $e_v,e_w$ share a vertex
(see Observation~\ref{Obs:Sharing a vertex}). Hence, it must preserve the identification of a ``copy'' of the boundary
of $\mathrm{Conv}(P)$ in $G(P)$
presented in Section~\ref{sec:sub:boundary} (which is defined up to an automorphism of $K(P)$). Finally, it follows
from the recovery process presented in Sections~\ref{sec:sub:levels} and~\ref{sec:sub:diagonals} that an automorphism
of $G(P)$ is completely determined by its action on the copy of the boundary of $\mathrm{Conv}(P)$ in $G(P)$. Therefore, $\mathrm{Aut}(G(P)) \cong D_{n}$.

\section{Complexity Analysis}
\label{sec:analysis}

The algorithmic approach presented in the previous sections allows us not only to show that
$\mathrm{Aut}(G(P))\cong D_{n}$, but also to recover the edges of all paths represented by
vertices of $G(P)$ efficiently. The following theorem calculates the complexity of our
algorithm.

\begin{theorem} \label{thm:complexity}
Let $G(P)$ be the path graph of a set $P$ of $n \geq 5$ points in convex position in the plane, and
denote $N:= |V(G(P))| = n2^{n-3}$ (see~\cite{AIM07}). The edges of all paths represented by
vertices of $G(P)$ can be recovered in time $O(N \log N)$.
\end{theorem}

\noindent We note that this complexity is not far from optimal, since the graph $G(P)$ contains $N$ vertices, and its recovery requires identifying
the path of size $n-1 \approx \log N$ that each vertex represents.
In the proof of the theorem we will use an auxiliary lemma. Recall that by Claim~\ref{claim:deg}, the degree of each vertex $v$ in $G(P)$ is at most $O(n)=O(\log N)$. The lemma asserts that the average degree is much smaller - namely, bounded by a constant.

\begin{lem} \label{Lem:linear_edges}
$|E(G(P))|=O(N).$
\end{lem}

The proof of the lemma will be presented at the end of this section, and meanwhile we present the proof of the theorem.

\begin{proof}[of Theorem~\ref{thm:complexity}]
We go over the steps of the algorithm that recovers the edges of all paths and calculate the complexity
of each step.

\medskip \noindent \textbf{Recovery of the boundary paths.}
As mentioned in Section~\ref{sec:sub:boundary}, identifying the set $\B$ of all boundary paths as a set, can be done by finding the vertices of degree $3n-7$ in $G(P)$.
The complexity of this step is $$\sum_{v \in V(G(P))}\mathrm{deg}(v) = 2|E(G(P))|=O(N),$$ using Lemma~\ref{Lem:linear_edges}.

\medskip \noindent \textbf{Detecting a ``copy'' of the boundary of $\mathrm{Conv}(P)$ in $G(P)$.} As mentioned in Section~\ref{sec:sub:boundary}, once the
set $\B$ of vertices that represent the $n$  boundary paths is found,
 this step can be performed easily by going over all edges spanned by pairs of vertices in
 $\B$
and checking whether each such edge is contained in a max-clique of size 4 or not.
By Corollaries~\ref{cor:max} and~\ref{cor:clique size}, for each such pair ${u,v}$, it is sufficient to check whether there exists $w \in V(G(P)) \setminus \B$ which is a common neighbor of $u$ and $v$. Since the number of neighbors of any vertex in $G(P)$ is bounded by $O(\log N)$,
the complexity of this step is less than $O(\log^4 N)$ operations.

\medskip \noindent \textbf{Recovering all edges of each path.}

We prove that
this third step can be performed in $O(N \log N)$
operations, using the following strategy. For each $v \in V(G(P))$, we store three types of information:
\begin{enumerate}
\item $\mathcal{B}(v)$ (i.e., the set of boundary edges of $P(v)$),

\item $\ell(v)$ (i.e., the level of $v$),

\item The endpoints of $P(v)$.
\end{enumerate}
Note that by the proof of Observation~\ref{Obs:Completion by diagonals}, items~(1)--(3) yield full
recovery of the edges of $\mathcal{P}(v)$.

We go over the vertices of $G(P)$ by levels, starting with level~0, then level~1 (i.e., the neighbors
of the vertices in level~0 that were not dealt with yet), then level~2, etc.

For each $v \in V(G(P))$ with $\ell(v)=i \geq 2$, items~(1)--(3) for $v$ can be computed instantly given
items~(1)--(3) for all neighbors of $v$ at level~$(i-1)$ (as described in Observation~\ref{Obs:LevelBeyond1} and in the proof of Observation~\ref{Obs:Level-recovery}).

For vertices with $\ell(v)=1$, recovery of item~(3) requires the knowledge of items~(1)--(2) for
their neighbors at levels~0,2 (as described in Observation~\ref{Obs:Level1}). Hence, after computing
items~(1)--(3) for all vertices at level~0, we compute items~(1)--(2) for the vertices at level~1, then
items~(1)--(2) for vertices at level~2, then item~(3) for vertices of level~1, then item~(3) for vertices
at level~2, and then all items in increasing order of levels.


The treatment of each vertex $v$ requires going over each neighbor $u$ of $v$, and (in the worst case) reading the information-type $\B(u)$ whose size is at most $n-1$. Eventually, each edge of $G(P)$ is considered twice, where each treatment requires $O(n)$ operations, and thus, by Lemma~\ref{Lem:linear_edges}, the total number of operations is bounded by $O(nN)=O(N \log N)$.
Therefore, the total time complexity of our algorithm is $O(N \log N)$, as asserted.
\qed
\end{proof}

Now, it only remains to prove Lemma~\ref{Lem:linear_edges}.

\begin{proof}[of Lemma~\ref{Lem:linear_edges}]
By symmetry, we may consider the vertices $v$ of $G(P)$ that correspond to paths $P(v)$ in which one leaf $x_0 \in P$ is fixed, and then multiply the result by $n$.
We represent any such path
$P(v)=\langle x_0, x_1, \ldots x_{n-1} \rangle$ by a binary vector $\langle \alpha_0, \alpha_1, \ldots \alpha_{n-1} \rangle$ where $\alpha_0=0$, and
$\alpha_i=0$ if and only if the edge $(x_{i-1},x_i)$ in $P(v)$ is a boundary edge. Note that $\alpha_{n-1}=0$. Assume that $P(v)$ contains at least two diagonals. We call $P(v)$ \emph{a path of type $t$} if $$t=\min_i{\{\alpha_i \neq 0 \}} + \min_j{\{\alpha_{n-j} \neq 0 \}},$$ namely, if $P(v)$ starts with $k-1$ boundary edges and ends with $l-1$ boundary edges, for some $k,l \geq 2$ such that $k+l=t$.

We observe that a neighbor of $v$ in $P(v)$ can be obtained only by deleting one out of the first $k$ edges or the last $l$ edges of $P(v)$, and adding another edge instead. Indeed,  deletion of any other edge of $P(v)$ decomposes $P(v)$ into two paths, where a leaf of the first one is $x_0$, a leaf of the second one is $x_{n-1}$, the diagonal $(x_{k-1},x_k)$ belongs to the first path and separates $x_0$ from the second path, and the diagonal $(x_{n-l-1},x_{n-l})$ belongs to the second path and separates $x_{n-1}$ from the first path. Therefore, there does not exist any edge that can be added to the union of these two paths in order to form a simple path (except for the deleted edge).

On the other hand, for any deletion of one of the first $k$ edges or one of the last $l$ edges of a path $P(v)$ of type $t$, there exist at most 4 edges that can be added to the union of the two paths in order to obtain a Hamiltonian path. Hence, the number of neighbors of $v$ in $G(P)$ is bounded by $3(k+l)=3t$.

In addition, for any path of type $t$ there are $t-3$ possible choices of $k,l$ as above, and thus, the number of paths of type $t$ whose endpoint is $x_0$ is bounded by $O(t \cdot 2^{n-1-t})$, which implies that the total number of paths of type $t$ is bounded by $O(n \cdot t \cdot 2^{n-1-t})$. To conclude, the number of edges of $G(P)$ of the form $(v,v')$ where $P(v)$ is a path of type $t$ with at least two diagonals, is bounded by
$O(n \cdot t^2 \cdot 2^{n-1-t})$.

The number of edges of $G(P)$ of the form $(v,v')$ where $P(v)$ and $P(v')$ are paths that contain at most one diagonal, is bounded by $O(n^3)$ and thus is negligible.

Putting things together, the number of edges in $G(P)$ is at most
\begin{align*}
|E(G(P))| &=O \left(\sum_{t=4}^{n-1} n \cdot t^2 \cdot 2^{n-1-t} \right) = O \left(n \cdot 2^{n-1} \cdot \sum_{t=4}^{n-1}\frac{t^2}{2^t} \right) \\
&\leq O \left(n \cdot 2^{n-1} \cdot \sum_{t=0}^{\infty}\frac{(t+2)(t+1)}{2^t} \right) = O(n \cdot 2^{n-1}) \\
&= O(N),
\end{align*}
where the penultimate equality follows from the well-known equality $$\sum_{t=0}^{\infty}\frac{(t+2)(t+1)}{2^t} = 16,$$ that can be easily proved by differentiating twice the series $\sum_{t=0}^{\infty}x^t$ and substituting $x=0.5$.
\qed
\end{proof}


\section*{Open Problems}
\label{sec:open-problems}

\medskip \noindent We conclude this paper with a few questions for further research that stem from our results.

\medskip \noindent \textbf{The automorphism group of other subgraphs of $\mathcal{T}(P)$.}
In~\cite{Hernando}, Hernando showed that if $P$ is a set of points in convex position and $\T(P)$ is its
geometric tree graph, then $\mathrm{Aut}(\T(P)) \cong D_{n}$, as we showed for $G(P)$. In view of the fact that
$G(P)$ is a subgraph of $\T(P)$, it is reasonable to ask whether $\mathrm{Aut}(G'(P)) \cong D_{n}$ holds also
for other subgraphs $G(P) \subset G'(P) \subset \T(P)$. For example, does this hold for the graph of simple
spanning trees with maximal degree $\leq d$?

\medskip \noindent \textbf{Points in general position.}
What can be said if the points of $P$ are in general (rather than convex) position? Can we prove that
$\mathrm{Aut}(G(P)) \cong \mathrm{Aut(K(P))}$?

\medskip \noindent \textbf{Abstract graphs.} What happens in the abstract case? That is, if $G'(P)$ is the path graph of abstract $K(P)$, is this true that $\mathrm{Aut}(G'(P)) \cong \mathrm{Aut}(K(P)) \cong S_n$?
It was shown in~\cite{KP15+} that this holds for the tree graph of $K(P)$.

\old{
\begin{enumerate}

\item In~\cite{Hernando}, Hernando showed that if $P$ is a set of points in convex position and $\T(P)$ is its
geometric tree graph, then $\mathrm{Aut}(\T(P)) \cong D_{n}$, as we showed for $G(P)$. In view of the fact that
$G(P)$ is a subgraph of $\T(P)$, it is reasonable to ask whether $\mathrm{Aut}(G'(P)) \cong D_{n}$ holds also
for other subgraphs $G(P) \subset G'(P) \subset \T(P)$. For example, does this hold for the graph of plane
spanning trees with maximal degree $d$?

\item What can be said if the points of $P$ are in general (rather than convex) position? Can we prove that
$\mathrm{Aut}(G(P)) \cong \mathrm{Aut(K(P))}$? Is this true at least that any automorphism of $K(P)$ induces
an automorphism of $G(P)$?

\item What happens in the abstract case? That is, if $G'(P)$ is the path graph of abstract $K(P)$,
is this true that $\mathrm{Aut}(G'(P)) \cong \mathrm{Aut}(K(P)) \cong S_n$? It was shown in~\cite{KP15+} that
this holds for the tree graph of $K(P)$.

\end{enumerate}
}


\subsection*{Acknowledgments}
\label{sec:acknowledgments}

The authors are grateful to Gila Morgenstern for her contribution in the first steps of this research.


%


%

\end{document}